\documentclass[12pt,letterpaper,pagesize]{amsart}
\usepackage[utf8]{inputenc}
\usepackage{amsmath}
\usepackage{amsthm}[2015/03/10 v2.20.2]
\numberwithin{equation}{section}
\usepackage{amsfonts}
\usepackage{afterpage}
\usepackage{amssymb}
\usepackage{graphicx}
\usepackage{pgf}
\usepackage{pgfplots}
\pgfplotsset{width=10cm,compat=1.9}
\usepackage[all]{xy}
\usepackage{tikz}
\usepgfplotslibrary{external}
\usetikzlibrary{pgfplots.external}
\usepackage[left=2.5cm,right=2cm,top=2cm,bottom=2cm]{geometry}
\newtheorem{theorem}{Theorem}[section]

\theoremstyle{definition}
\newtheorem{definition}[theorem]{Definition}

\newtheorem{prop}[theorem]{Proposition}

\newtheorem{prob}{Problem}
\newtheorem{obs}[theorem]{Remark}
\numberwithin{equation}{section}
\usepackage{hyperref}

\newtheorem*{acknowledgements}{Acknowledgements}

\title{Gears and pulleys in non-Euclidean space forms}
\author{Heleno S. Cunha, Lucas H. R. de Souza and Sergio A. P. Prado}

\begin{document}

\DeclareGraphicsExtensions{.pdf,.jpg,.mps,.png,}

\maketitle

\def\eod{\hfill$\square$}

\begin{abstract}In this article we study, in their non-Euclidean versions, two important mechanical systems that are very common in numerous devices. More precisely, we study the laws governing the movement of pulley and gear systems in spherical and hyperbolic geometries. And curiously, we were able to see an interesting similarity between the determined laws.
\end{abstract}

\let\thefootnote\relax\footnote{Mathematics Subject Classification (2010). Primary: 51P05, 70B10; Secondary: 51M10.}
\let\thefootnote\relax\footnote{Keywords: Gears, pulleys, hyperbolic geometry, spherical geometry, space forms.}

\section*{Introduction}

Let's consider the following elementary problems in mechanics, involving gears and pulleys:

\begin{prob}\label{gearprob}Let's consider two gears attached to each other. a) If we rotate one gear at some angular speed, what is the angular velocity of the second gear, depending on the angular velocity of the first one and the numbers of teeth of both gears? b) If we rotate one gear at some angular speed, what is the angular velocity of the second gear, depending on the angular velocity of the first one and the radii of both gears?
\end{prob}

\begin{figure}[h]
\centering
\includegraphics[scale=0.12]{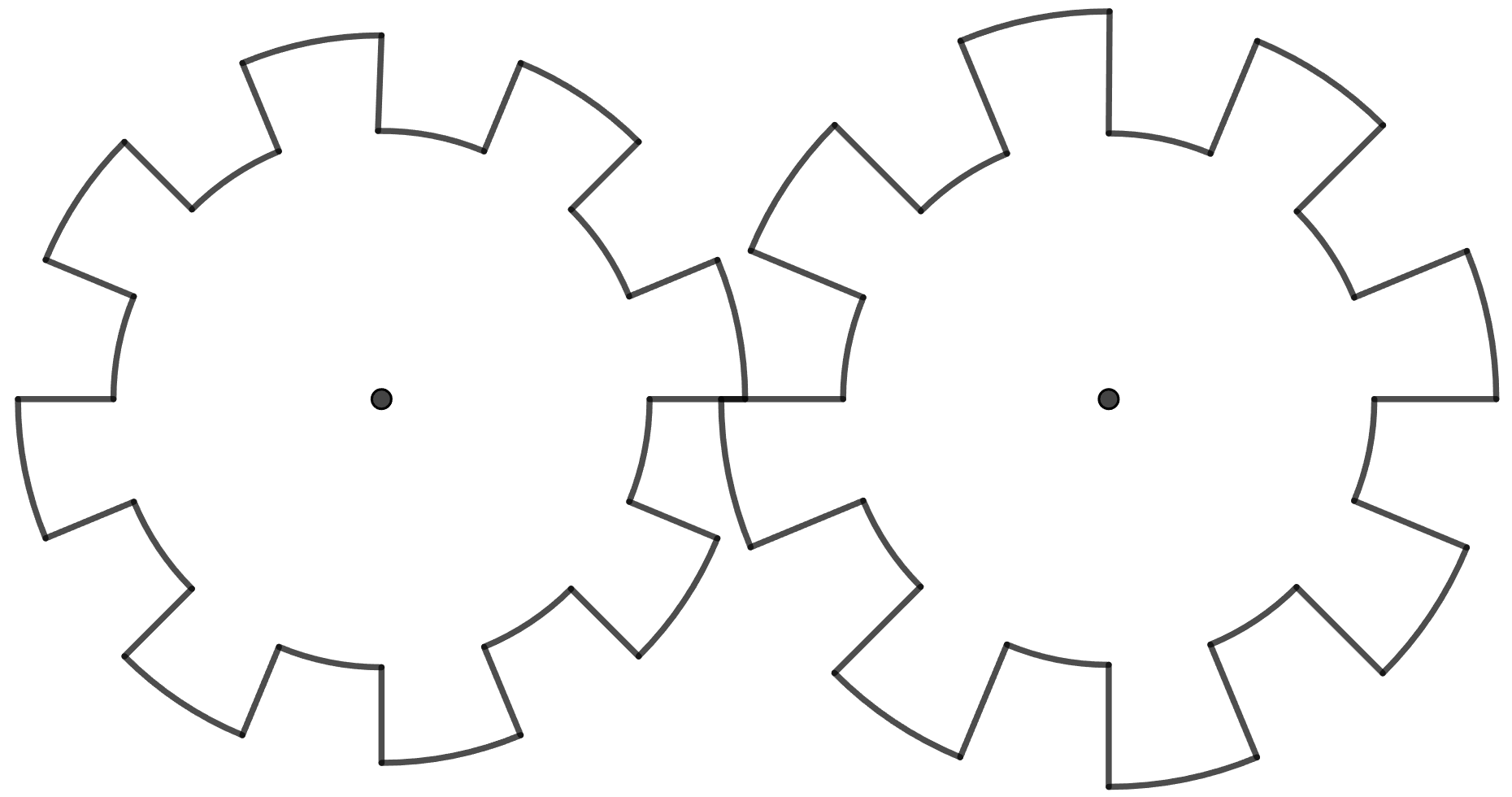}
\caption{\small A system of two gears in the Euclidean space attached to each other.}
\end{figure}

\begin{prob}\label{pulleyprob}Let's consider two pulleys attached by a tensioned belt. If we rotate one pulley at some angular speed, then it moves the belt, which transports the movement to the second pulley. Let's consider that there is no sliding motion in the process. What is the angular velocity of the second pulley, depending on the angular velocity of the first one and the radii of both pulleys?
\end{prob}

\begin{figure}[h]
\centering
\includegraphics[scale=0.1]{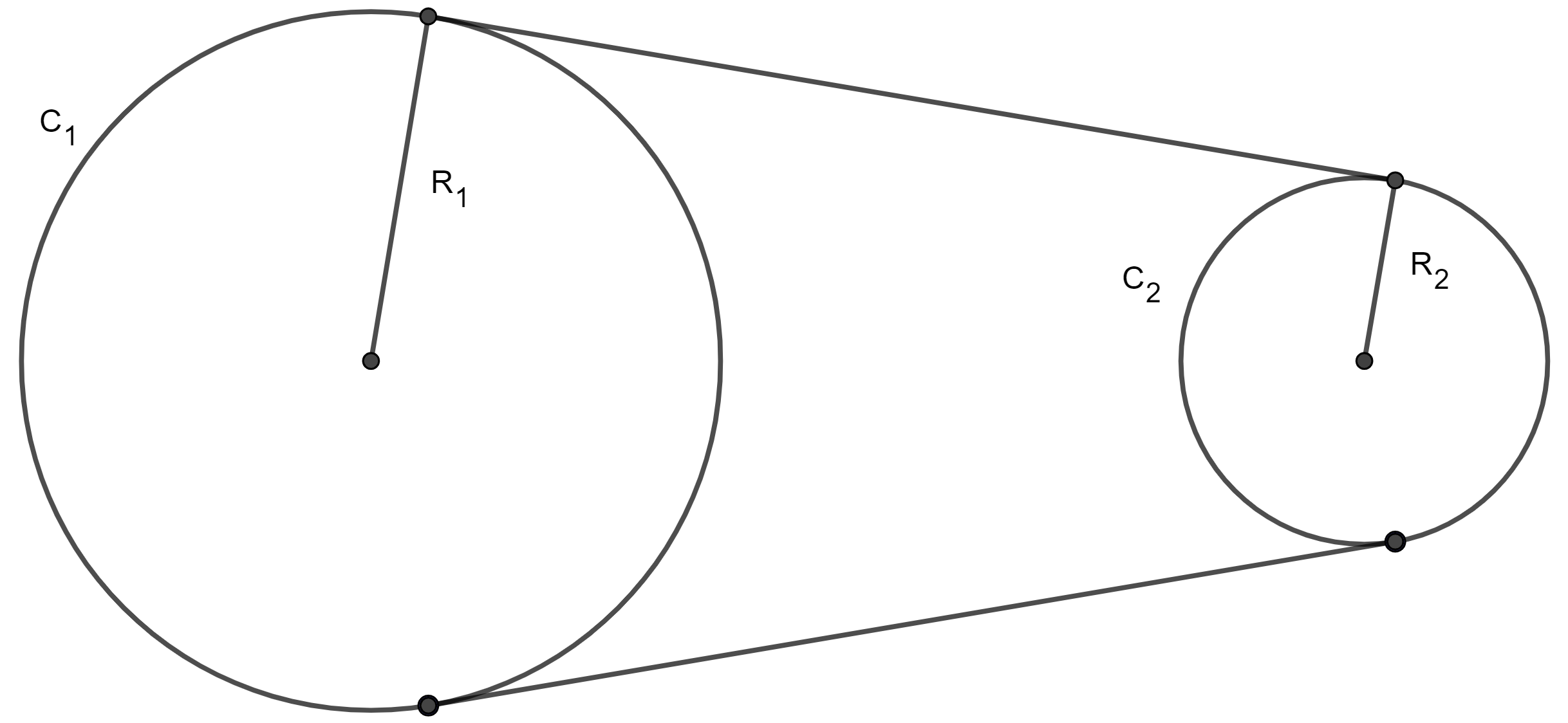}
\caption{\small A system of two pulleys in the Euclidean space attached by a tensioned belt.}
\end{figure}

Removing all of the physical characteristics of these problems, we get that it turns out to be two problems in Euclidean Geometry. These two problems have different natures, but with similar answers.

If the two gears have, respectively, $n_{1}$ and $n_{2}$ teeth and their angular velocities are $\omega_{1}$ and $\omega_{2}$, respectively, then we have the following relation (Equation (5) of page 70 of \cite{Pa}):
$$ n_{1}\omega_{1} = n_{2}\omega_{2}$$

If the two pulleys or two gears have radii $R_{1}$ and $R_{2}$ and angular velocities $\omega_{1}$ and $\omega_{2}$, respectively, then we have the following relation (Equation (2) of page 69 of \cite{Pa} for gears and page 331 of \cite{KG} for pulleys):
$$R_{1}\omega_{1} = R_{2}\omega_{2}$$

We want to see what happens if we transport both problems to Hyperbolic Geometry and Spherical Geometry. We have:

\

\textbf{Theorem \ref{teeth}} Let $C_{1}$ and $C_{2}$ be two gears attached to each other in the hyperbolic plane or in the sphere of radius $1$. If the two gears have, respectively, $n_{1}$ and $n_{2}$ teeth and angular velocities are $\omega_{1}(t)$ and $\omega_{2}(t)$, respectively, then we have the following relation:
$$n_{1}\omega_{1}(t) = n_{2}\omega_{2}(t) $$

\

\textbf{Theorems \ref{hyperbolicgear}} and \textbf{\ref{hyperbolic}} Let $C_{1}$ and $C_{2}$ be two gears attached to each other or two pulleys attached by a tensioned belt in the hyperbolic plane. If $C_{i}$ have radius $R_{i}$ and angular velocity  $\omega_{i}(t)$, then we have the following relation:

$$\sinh(R_{1}) \omega_{1}(t) =  \sinh(R_{2})\omega_{2}(t)$$

\begin{figure}[h]
\centering
\includegraphics[scale=0.1]{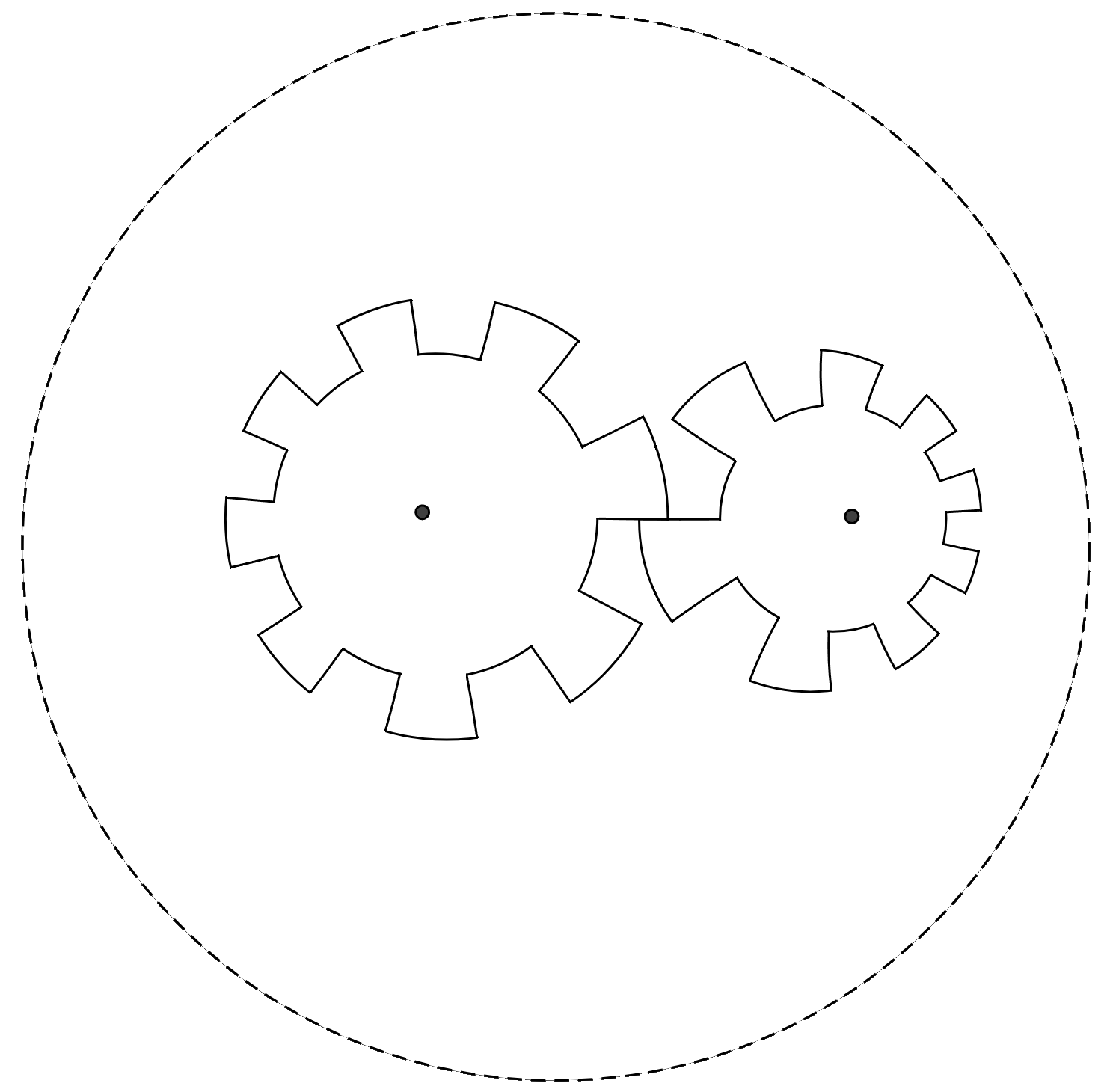} \ \ \
\includegraphics[scale=0.3]{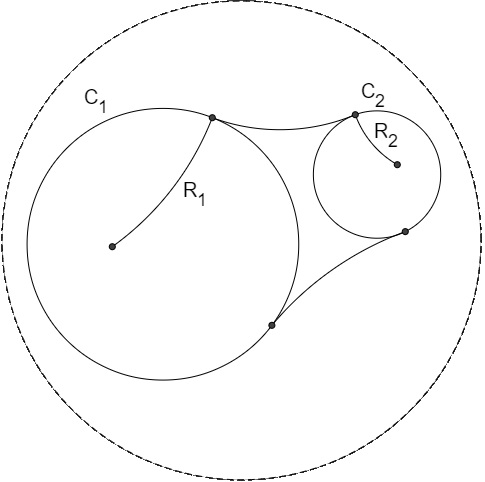}
\caption{\small The first image is a representation of two gears in the Poincaré-disk model of the hyperbolic plane. The representation of the gears' teeth is merely illustrative, as considered in \textbf{Section 1.1}. The second image is a representation of two pulleys in the Poincaré-disk model of the hyperbolic plane attached by a tensioned belt that satisfies the conditions of the theorem above. The belt is represented here by two geodesics that are tangent to both circles and it is merely illustrative.}
\end{figure}

\

\textbf{Theorems \ref{sphericalgear}} and \textbf{\ref{spherical}} Let $C_{1}$ and $C_{2}$ be two gears attached to each other or two pulleys attached by a tensioned belt in the sphere of radius $1$. If $C_{i}$ have radius $R_{i}$ and angular velocity  $\omega_{i}(t)$, then we have the following relation:
$$\sin(R_{1}) \omega_{1}(t) =  \sin(R_{2})\omega_{2}(t)$$

\begin{figure}[h]
\centering
\includegraphics[scale=0.345]{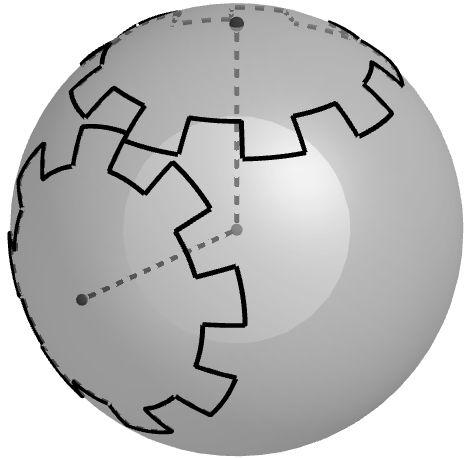} \ \ \
\includegraphics[scale=0.3]{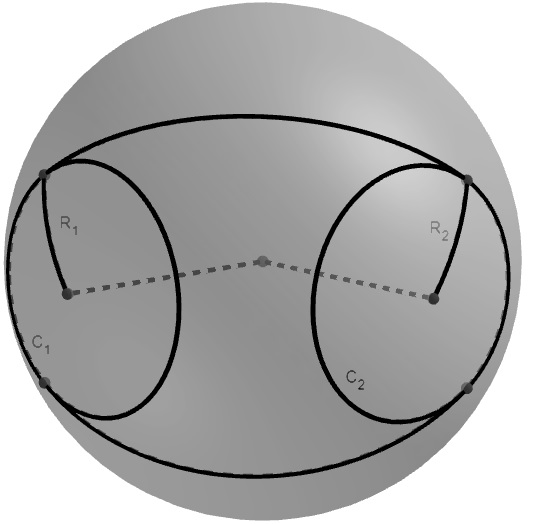}
\caption{\small The first image is a representation of two gears in the sphere. The representation of the gears' teeth is merely illustrative, as considered in \textbf{Section 1.1}. The second image is a representation of two pulleys in the sphere attached by a tensioned belt that satisfies the conditions of the theorem above. The belt is represented here by two geodesics that are tangent to both circles and it is merely illustrative.}
\end{figure}

\

There is another problem whose already well known solutions have similar behavior, when we change the geometry: take two point particles $p_{1}$ and $p_{2}$ with masses $m_{1}$ and $m_{2}$, respectively. If $p$ is the centroid of the system of particles given by these two points, then we get the following relations (called rules of lever) \cite{Ga}:

\begin{enumerate}
    \item In the euclidean plane: $d_{\mathbb{E}^{2}}(p_{1},p) m_{1} = d_{\mathbb{E}^{2}}(p_{2},p) m_{2}$
    \item In the hyperbolic plane: $\sinh(d_{\mathbb{H}^{2}}(p_{1},p)) m_{1} = \sinh(d_{\mathbb{H}^{2}}(p_{2},p)) m_{2}$
    \item In the sphere of radius $1$: $\sin(d_{\mathbb{S}^{2}}(p_{1},p)) m_{1} = \sin(d_{\mathbb{S}^{2}}(p_{2},p)) m_{2}$
\end{enumerate}

This paper is composed of three sections. In the first section we take the two problems presented above, which have physical nature, and we abstract them and formalize them to become geometric problems. We also give a proof to \textbf{Theorem \ref{teeth}} that doesn't depend on the choice of the geometry. For the  other theorems presented above, it is necessary to choose the geometry first, so the second and the third sections are devoted to briefly present the models of the geometries that we need and to do the proofs of the theorems in the context of hyperbolic geometry and spherical geometry, respectively. For the \textbf{Theorem \ref{hyperbolic}} we give two proofs (in \textbf{Sections 3.3} and \textbf{3.5}). For the first one we use the Poincaré-disk model and for the second one we use the hyperboloid model.

\begin{acknowledgements}
The authors would like to thank Solange Schardong for the discussions about this paper.

The second author was partially supported by the joint FAPEMIG/CNPq program "Programa de Apoio à Fixação de Jovens Doutores no Brasil", Grant Number BPD-00721-22 (FAPEMIG), 150784/2023-6 (CNPq).

The third author was partially supported by the CNPq, Grant number 153794/2022-4.
\end{acknowledgements}

\section{Transforming the physical problems into geometric problems}

\subsection{Gears}

A gear is a simple closed path $C$ in a space form $X$ with a finite set of teeth. Each tooth has the same length in $C$. When we have two gears attached, then one tooth of it is attached to a tooth of the other one.  If the first one is in movement, then its tooth $T$ pushes the tooth of the other gear that is attached to it to the next space, causing the movement of the second gear. We are not interested in what is the nature of the teeth or how it pushes the teeth of the other gear. It could be mechanical, magnetic, or other physical way of doing this. We care just that the realized movement is the same as the described above.

To be a geometric problem, we need to formalize what we mean by a gear, its teeth, rotation of the gear, speed of the gear, how the movement of one gear rotates the other one and these definitions need to make sense in all three space forms.

We fix a space form $X$. We consider that a gear $C$ is a geometric circumference on $X$, together with a division of it in arcs $\{T^{1}, T^{2}, ..., T^{n}\}$ of the same length (written in cyclic order). The arcs $T^{i}$ are called the teeth of $C$. The order of the teeth gives an orientation for $C$.

A rotation movement of a gear is a map $(r,f): \{0,...,s\} \rightarrow Isom(X) \times \mathbb{Z}$ such that $r(0) = id_{X}$, $f(0) = 0$ and for every $t \in \{0,...,s\}$, $r(t)$ is a rotation of the space form $X$ with the center of the gear as its axis of rotation such that for any tooth $T$, $r(t)(T)$ is also a tooth of the gear. The second coordinate says how many full rotations we have in that moment of time.

Let's fix a rotation movement $r$ of a gear $C$ with center $c$. If $p \in X$, then we have a map $r_{p}: \{0,...,s\} \rightarrow X$ given by $r_{p}(t) = r(t)(p)$. Then there is an associated winding map $\sigma_{p}: \{0,...,s\} \rightarrow \mathbb{Z}$ defined by $\sigma_{p}(t) = \frac{\delta(p,r(t)(p))}{a(T)}+nf(t)$, where $a(T)$ is the arc length given by any tooth $T$, $\delta(p,r(t)(p))$ is the oriented angle (with respect to the orientation of $C$) between $\overline{c,p}$ and $\overline{c,r(t)(p)}$ and $n$ is the number of teeth of $C$.

\begin{prop} Let $C$ be a gear with center $p \in C$ and $\sigma: \{0,...,s\} \rightarrow \mathbb{Z}$ a map. Then there exist a unique rotation movement such that its associated winding map with respect to $p$ is $\sigma$.
\end{prop}

\begin{proof}We have that, for every $t \in \{0,...,s\}$, we can write $\sigma$ as  $\sigma(t) = g(t) + n f(t)$, for some maps $g: \{0,...,s\} \rightarrow \{0,...,n-1\}$ and $f: \{0,...,s\} \rightarrow \mathbb{Z}$ and this decomposition is uniquely defined (by the Euclidean division algorithm). So we are able to define a rotation movement $(r,f):\{0,...,s\} \rightarrow Isom(X)$, where $r(t)$ is a the unique rotation with center $c$ in the direction of the orientation of $C$ such that $\frac{\delta(p,r(t)(p))}{a(T)} = g(t)$. It follows that $\sigma_{p} = \sigma$. The uniqueness of $(r,f)$ is immediate from the uniqueness of the decomposition of $\sigma$.
\end{proof}

\begin{prop}The associated winding map does not depend on the choice of the point.
\end{prop}

\begin{proof}Let $p,q \in C$. Since for every $t \in \{0,...,s\}$, $r(t)$ is a rotation with center $c$, then $\delta(p,r(t)(p)) = \delta(q,r(t)(q))$, which implies that $\sigma_{p}(t) = \sigma_{q}(t)$.
\end{proof}

\begin{prop}The associated winding map is invariant by isometries.
\end{prop}

\begin{proof}Let $S$ be an isometry in $X$ and $p \in C$. Then $S \circ r(t) \circ S^{-1}$ is a rotation that fixes the point $S(c)$. We have that $\delta(S(p),S \circ r(t)\circ S^{-1}(S(p))) = \delta(S(p),S \circ r(t)(p)) = \delta(p, r(t)(p))$, which implies that the associated winding map $\sigma_{S(p)}$ of the rotation $(S\circ r(t) \circ S^{-1},f(t))$ of the gear $S(C)$ is equal to $\sigma_{p}$ (the associated winding map of the rotation $(r(t),f(t))$ of the gear $C$).
\end{proof}

We define the angular velocity of $C$ with respect to the movement $(r,f)$ as $\omega(t) = \frac{2\pi\sigma(t)}{n t}$ (where $\sigma$ is the associated winding map).

\begin{prop}The angular velocity is invariant by isometries.
\end{prop}

\begin{proof}Immediate from the last proposition.
\end{proof}

Let $C_{1}$ and $C_{2}$ be two gears such that the length of the teeth of both are equal. We say that they are attached if every rotation of one gear is always transferred to a rotation of the other gear in a way that preserves their associated winding maps.

\begin{theorem}\label{teeth} Let $C_{1}$ and $C_{2}$ be two gears attached to each other in the space form $X$. If the two gears have, respectively, $n_{1}$ and $n_{2}$ teeth and their angular velocities are $\omega_{1}$ and $\omega_{2}$, respectively, then we have the following relation:
$$n_{1}\omega_{1}=n_{2}\omega_{2} $$
\end{theorem}

\begin{proof} Let's suppose that the movements of the gears $C_{1}$ and $C_{2}$ have, respectively, winding maps $\sigma_{1}$ and $\sigma_{2}$. Since these gears are attached to each other, then $\sigma_{1} = \sigma_{2}$. We also have that $\frac{2 \pi \sigma_{i}(t)}{t} = n_{i}\omega_{i}(t)$, which implies that  $n_{1}\omega_{1}(t) =  n_{2}\omega_{2}(t)$.
\end{proof}

Which gives an interpretation of \textbf{Problem \ref{gearprob} a)}

Then we get the following problem:

\

\textbf{Geometric version of Problem \ref{gearprob} b)} Let's consider two attached gears $C_{1}$ and $C_{2}$ with radii $R_{1}$ and $R_{2}$, respectively, that are rotating. What is the relation between the angular velocity of the gears, depending on the radii of both gears?

\

\subsection{Pulleys}

Let's consider two pulleys attached by a tensioned belt. If we rotate one of these pulleys at some angular velocity $\omega_{1}$, then we get some linear velocity $v$ of the points of the boundary of the pulley. Let's consider that there is no sliding in the whole movement. Then, the belt starts to move with the same linear velocity $v$, which forces the second pulley to move with some angular velocity $\omega_{2}$ in such a way that the linear velocity of the points in the boundary of the second pulley are also $v$. Here, we do not care about what is a tensioned belt (but it is still an interesting question) and what is the nature of friction (i.e. why the velocity is transferred from the first pulley to the belt and from the belt to the second pulley).

To be a geometric problem, we need to formalize what we mean by a pulley, rotation of the pulley, linear velocity of a point and angular velocity and these definitions need to make sense in all three space forms.

Let's fix a space form $X$. By a pulley we mean a geometric circle inside $X$. A rotation movement of the pulley is a differentiable map $r: [0,s] \rightarrow Isom(X)$ such that $r(0) = id_{X}$ and for every $t \in[0,s]$, $r(t)$ is a rotation of the space form $X$ with the center of the pulley as its axis of rotation. Since the pulley is a circle and for every $t \in [0,s]$, $r(t)$ is a rotation whose center is the center of the pulley, then we have that the pulley is an invariant set of the "action" of the rotation movement (i.e., if $C$ is the pulley, then for every $t \in[0,s]$, $r(t)(C) = C$).

Let's fix a rotation movement $r$ of a pulley. If $p \in X$, then we have a differentiable map $r_{p}: [0,s] \rightarrow X$ given by $r_{p}(t) = r(t)(p)$.

The linear velocity of the point $p$ and the angular velocity of the pulley with respect to the center of the pulley are defined below.

Let $X$ be a space form, $\gamma: [0,s] \rightarrow X$ a differentiable path, $p \in X$ and $\lambda$ geodesic segment that starts at $p$.

\begin{definition}The linear velocity of $\gamma$ is defined by $\gamma'$. Its norm at the time $t$ is defined by $||\gamma'(t)||_{\gamma(t)}$.
\end{definition}

\begin{obs}\label{linearvelocityinvariant} Note that, by the definition of isometry, we have that the norm of the linear velocity is invariant by isometries.
\end{obs}

\begin{definition} Let $p \in X$ and $\lambda$ a geodesic segment that starts on $p$. Consider the map $\alpha_{p,\lambda}(\gamma): [0,s] \rightarrow \mathbb{R}$  where $\alpha_{p,\lambda}(\gamma)(t)$ is the angle between $\lambda$ and the geodesic segment $\overline{p,\gamma(t)}$ (we use the notation $\measuredangle (\lambda, \overline{p,\gamma(t)})$). We call the angular velocity of $\gamma$ with respect to the point $p$ as $\omega_{p}(\gamma)(t)$ defined by $|\alpha'_{p,\lambda}(\gamma)(t)|$.
\end{definition}

\begin{obs}Note that if $X = S^{2}$, then the map $\alpha_{p,\lambda}(\gamma)$ is not well defined if $\gamma$ passes through the antipodal point of $p$. This does not happen on the curves that we need to consider in this paper.
\end{obs}

\begin{prop}The angular velocity does not depend on the choice of $\lambda$.
\end{prop}

\begin{proof}Let $p$ be a point in $X$ and $\lambda_{1}$, $\lambda_{2}$ two segments that start from $p$. Then $\alpha_{p,\lambda_{1}}(\gamma) - \alpha_{p,\lambda_{2}}(\gamma)$ is constant equal to the angle  between $\lambda_{1}$ and $\lambda_{2}$. Then  $\alpha'_{p,\lambda_{1}}(\gamma) =  \alpha'_{p,\lambda_{2}}(\gamma)$, which implies that  $|\alpha'_{p,\lambda_{1}}(\gamma)| =  |\alpha'_{p,\lambda_{2}}(\gamma)|$.
\end{proof}

\begin{prop}\label{angularvelocityinvariant}The angular velocity is invariant by isometries of $X$.
\end{prop}

\begin{proof}Let $p \in X$, $\lambda$ a segment that starts at $p$ and $T$ an isometry of $X$. Then $\alpha_{p,\lambda}(\gamma)(t)$ is the angle between $\lambda$ and $\overline{p,\gamma(t)}$, which is the same angle as the angle between $T(\lambda)$ and $\overline{T(p),T \circ \gamma(t)}$ (which is $\alpha_{T(p),T(\lambda)}(T \circ \gamma)(t)$. Then  $\omega_{p}(\gamma) =  \omega_{T(p)}(T \circ \gamma)$.
\end{proof}

\begin{prop}\label{angularvelocitywelldefined}Let $m:[0,s] \rightarrow Isom(X)$ be a rotation of $X$ with center $c$ and $p,q \in X-\{c\}$ ($p,q$ also different than the antipodal point of $c$, if $X = \mathbb{S}^{2}$). Then $\omega_{c}(m_{p}) = \omega_{c}(m_{q})$.
\end{prop}

\begin{proof}Suppose that $p$ and $q$ have the same distance to the point $c$. Then there exists an isometry $T$ that fixes the point $c$ and such that $T(p) = q$. Then we have that $m_{q} = T \circ m_{p}$. By the last proposition, it follows that $\omega_{c}(m_{p}) = \omega_{c}(m_{q})$.

Suppose that $p$, $q$ and $c$ are collinear. Then $\measuredangle (\lambda, \overline{c,m_{p}(t)}) = \measuredangle (\lambda, \overline{c,m_{q}(t)})$, which implies that $\alpha_{c,\lambda}(m_{p}) = \alpha_{c,\lambda}(m_{q})$ and then $\omega_{c}(m_{p}) = \omega_{c}(m_{q})$.

Let's consider the general case of $p$ and $q$. Then there exists $p'$ such that $p, p'$ and c are collinear and the distance between $p'$ and $c$ is the same as the distance between $q$ and $c$. So $\omega_{c}(m_{p}) = \omega_{c}(m_{p'})$ and  $\omega_{c}(m_{p'}) = \omega_{c}(m_{q})$, which implies that  $\omega_{c}(m_{p}) = \omega_{c}(m_{q})$.
\end{proof}

So if $C$ is a pulley with center $c$ and $m$ is a rotation of $C$, then $\omega_{c}(m_{p})$ doesn't depend on the choice of the point $p$. So we can call it the angular velocity of $C$, when rotated by $m$.

Then we get the following problem:

\

\textbf{Geometric version of Problem \ref{pulleyprob}} Let's consider two pulleys $C_{1}$ and $C_{2}$ with radii $R_{1}$ and $R_{2}$, respectively. If both are rotating with the same linear velocity, then what is the relation between the angular velocity of the pulleys, depending on the radii of both pulleys?

\section{The hyperbolic plane}

\subsection{The Poincaré-disk model}

Let $\mathbb{D}^2=\lbrace z\in \mathbb{C}; |z|<1\rbrace$ be the unitary disc, and $\displaystyle ds^2=\frac{4|dz|^2}{(1-|z|^2)^2}$ be its Riemannian metric. By considering this convention, we recall that this space is known as the Poincaré-disk.  Then, if $p \in \mathbb{H}^{2}$ and $u,v \in T_{p}\mathbb{H}^{2}$ (we identify it with $\mathbb{R}^{2}$), the inner product is given by $\langle u,v\rangle_{p} = \frac{4\langle u,v\rangle}{(1-||p||^{2})^{2}}$, where $\langle,\rangle$ is the regular inner product of $\mathbb{R}^{2}$ and $||\cdot||$ its norm. We denote the norm of $T_{p}\mathbb{H}^{2}$ by $||\cdot||_{p}$.

If $C$ is a circle with center $0$, hyperbolic radius $R$ and Euclidean radius $r$, then we have the relations  $\sinh(\frac{R}{2}) = \frac{r}{(1-r^{2})^{1/2}}$ and $\tanh(\frac{R}{2}) = r$ (this relations are in the proof of Theorem 7.2.2 of \cite{Bea}). The length of this circle is $2\pi \sinh{R}$ (Theorem 7.2.2 (ii) of \cite{Bea}).

\subsection{Gears at the Poincaré-disk}

Let $C_{1}$ and $C_{2}$ be two gears attached to each other in the hyperbolic plane. Suppose that the two gears have, respectively, $n_{1}$ and $n_{2}$ teeth and hyperbolic radii $R_{1}$ and $R_{2}$, respectively. The movement of the gears' system depends on the width of the teeth, so assume that $a_{1}(T) = a_{2}(T)$. Notice that a circle $C_{i}$ at the hyperbolic plane with hyperbolic radius $R_{i}$ has length $2\pi\sinh(R_{i})$.

    Since $n_{i}a_{i}(T)= 2\pi\sinh(R_{i})$, then

    \begin{align*}
     \frac{2\pi\sinh(R_{1})}{n_{1}} = a_{1}(T) = a_{2}(T) = \frac{2\pi\sinh(R_{2})}{n_{2}}
    \end{align*}
    so $$\sinh(R_{1})n_{2}=\sinh(R_{2})n_{1}.$$

\begin{theorem}\label{hyperbolicgear} Let $C_{1}$ and $C_{2}$ be two gears attached to each other in the hyperbolic plane. If the two gears have radii $R_{1}$ and $R_{2}$ and angular velocities are $\omega_{1}$ and $\omega_{2}$, respectively, then we have the following relation:
$$\sinh(R_{1}) \omega_{1}(t) =  \sinh(R_{2})\omega_{2}(t)$$
\end{theorem}

\begin{proof}We have that $\frac{\sinh{R_{1}}}{\sinh{R_{2}}} = \frac{n_{1}}{n_{2}}$ and $\frac{n_{1}}{n_{2}} = \frac{\omega_{2}(t)}{\omega_{1}(t)}$ (by \textbf{Theorem \ref{teeth}}), which implies that $\frac{\sinh{R_{1}}}{\sinh{R_{2}}} = \frac{\omega_{2}(t)}{\omega_{1}(t)}$ and then $\sinh(R_{1}) \omega_{1}(t) =  \sinh(R_{2})\omega_{2}(t)$.
\end{proof}

\subsection{Pulleys at the Poincaré-disk}

Let $\alpha: [0,s] \rightarrow \mathbb{R}$ be a differentiable map that satisfies $\alpha(0) = 0$ and let $\gamma_{r,\alpha}: [0,s] \rightarrow \mathbb{H}^{2}$ be the differentiable curve defined by $\gamma_{r,\alpha}(t) =$ $r(cos(\alpha(t)),sin(\alpha((t))))$. It describes the movement over time of the point $p = (r,0)$ inside the circle of center $0$, euclidean radius $r$ and hyperbolic radius $R$ such that  $\sinh^{2}(\frac{R}{2}) = \frac{r^{2}}{1-r^{2}}$. The linear velocity of $p$ is given by $\gamma_{r,\alpha}'(t) = r\alpha'(t)(-sin(\alpha(t)),cos(\alpha(t)))$.

If we take $\lambda$ as the geodesic segment $\overline{0,(r,0)}$, then we have that $\alpha_{0,\lambda}(\gamma_{r,\alpha}) = \alpha$. This implies that the angular velocity of $\gamma_{r,\alpha}$ is given by $\omega_{0,\lambda}(\gamma_{r,\alpha}) = |\alpha'|$.

Let $C_{1}$ and $C_{2}$ be two gears or two pulleys in $\mathbb{H}^{2}$ and with hyperbolic radii $R_{1}$ and $R_{2}$, respectively. Suppose that both are rotating with the same norm for their linear velocities. Since the linear velocity and the angular velocity are invariant by isometries, we can suppose that $C_{1}$ and $C_{2}$ are centered in $0$. The euclidean radii of $C_{1}$ and $C_{2}$ are, respectively, given by the relation $\sinh^{2}(\frac{R_{i}}{2}) = \frac{r_{i}}{1-r_{i}^{2}}$. Then the point $(r_{i},0)$ is in $C_{i}$ and its movement is described by the curve $\gamma_{i} = \gamma_{r_{i},\alpha_{i}}$, for some map $\alpha_{i}$.

The norms of the linear velocities of $C_{1}$ and $C_{2}$ are the same, i.e. for every $t \in [0,s]$:
$$||\gamma_{r_{1},\alpha_{1}}'(t)||_{\gamma_{r_{1},\alpha_{1}}(t)} = ||\gamma_{r_{2},\alpha_{2}}'(t)||_{\gamma_{r_{2},\alpha_{2}}(t)}$$

Then we have:

$$\frac{4\langle \gamma_{r_{1},\alpha_{1}}'(t), \gamma_{r_{1},\alpha_{1}}'(t) \rangle}{(1-||\gamma_{r_{1},\alpha_{1}}(t)||^{2})^{2}} =  \frac{4\langle \gamma_{r_{2},\alpha_{2}}'(t), \gamma_{r_{2},\alpha_{2}}'(t) \rangle}{(1-||\gamma_{r_{2},\alpha_{2}}(t)||^{2})^{2}}$$
$$\frac{r^{2}_{1}\alpha_{1}'(t)^{2}}{(1-r_{1}^{2})^{2}} =  \frac{r^{2}_{2}\alpha_{2}'(t)^{2}}{(1-r_{2}^{2})^{2}} $$

Since $\sinh^{2}(\frac{R_{i}}{2}) = \frac{r_{i}^{2}}{1-r_{i}^{2}}$ and $\tanh^{2}\left(\frac{R_{i}}{2}\right)=r_{i}^{2}$, then:
$$\frac{r_{i}^{2}}{(1-r_{i}^{2})^{2}}=\frac{\sinh^{2}\left( \frac{R_{i}}{2}\right)}{1-\tanh^{2}\left(\frac{R_{i}}{2}\right)}=\frac{\sinh^{2}\left(\frac{R_{i}}{2} \right)}{\text{sech}^{2}\left(\frac{R_{i}}{2}\right)}=\left(\sinh\left(\frac{R_{i}}{2}\right)\cosh\left(\frac{R_{i}}{2}\right)\right)^{2} = (\frac{1}{2}\sinh(R_{i}))^{2};$$
so
$$\frac{1}{2}\sinh(R_{1})|\alpha_{1}'(t)| =  \frac{1}{2}\sinh(R_{2})|\alpha_{2}'(t)|,$$

since $\sinh(R_{i})$ is positive.

But $|\alpha_{i}'(t)| = \omega_{0}(\gamma_{i})(t)$. Then we have:
$$\sinh(R_{1}) \omega_{0}(\gamma_{1})(t) =  \sinh(R_{2})\omega_{0}(\gamma_{2})(t)$$

Since the angular velocity does not depend on the choice of the initial point (\textbf{Proposition \ref{angularvelocitywelldefined}}), then we have the theorem:

\begin{theorem}\label{hyperbolic} Let $C_{1}$ and $C_{2}$ be two pulleys attached by a tensioned belt in the hyperbolic plane. If $C_{i}$ has hyperbolic radius $R_{i}$ and angular velocity  $\omega_{i}$, then we have the following relation:
$$\sinh(R_{1}) \omega_{1}(t) =  \sinh(R_{2})\omega_{2}(t)$$ \eod
\end{theorem}

\subsection{The hyperboloid model}

Here we will use the hyperboloid model for $\mathbb{H}^{2}$. Therefore, we will make a brief presentation of it, but justifications for the facts presented and a more complete approach can be found in the books \cite{Iv} and \cite{Rat}. Note that this is another model of the hyperbolic plane, being isometric to the Poincaré-disk.

Let us denote by $\mathbb{R}^{2,1}$ the space $\mathbb{R}^3$ provided with a bilinear form, $\langle \ , \ \rangle: \mathbb{R}^3  \to \mathbb{R}$, that does not degenerate and from signature $(2,1)$. Let us adopt in $\mathbb{R}^{2,1}$ an orthonormal basis in which the bilinear form is given by
$$\langle p,q \rangle = p_1q_1 + p_2 q_2 - p_3 q_3.$$

Where $p = (p_{1}, p_{2}, p_{3})$ and $q = (q_{1}, q_{2}, q_{3})$. The set $\{ p \in \mathbb{R}^{2,1} \ |\ \langle p,p \rangle = -1 \}$ is the two-sheet hyperboloid. The hyperboloid model is the top sheet of the hyperboloid. Explicitly:
$$\mathbb{H}^2 = \{ p = (p_1, p_2, p_3) \in \mathbb{R}^{2,1} \ |\ \langle p,p \rangle = -1,\ p_3>0 \}.$$

It is a known fact that we can identify the tangent space of a point on the hyperbolic plane with its orthogonal complement:
$$T_p\mathbb{H}^2 = p^{\perp} = \{ v \in \mathbb{R}^{2,1} \ |\ \langle p,v \rangle =0 \}$$

The induced bilinear form on $T_{p}\mathbb{H}^{2}$ (i.e., $\langle u,v\rangle_{p} =\langle u,v\rangle$) is always positive defined. So we get a Riemannian metric for $\mathbb{H}^{2}$.  We denote the norm of $T_{p}\mathbb{H}^{2}$ by $||\cdot||_{p}$.

Given two points, $p$ and $q$ in $\mathbb{H}^2$, the hyperbolic metric between them, $d = d_{\mathbb{H}^2}(p,q)$, is given implicitly by:
$$\cosh (d) = - \langle p,q \rangle.$$

We record that the geodesic that connects two points, $p,q \in \mathbb{H}^2$ is the intersection of the subspace generated by $p,q$ in $\mathbb{R}^{2,1}$ with $\mathbb{H}^2$. Therefore, the geodesics in the hyperboloid model are branches of hyperbolas.

Let $C$ be a hyperbolic circumference in $\mathbb{H}^{2}$ with hyperbolic center $c_{H}$ and hyperbolic radius $R$. If $c_{H} = (0,0,1)$, then the hyperbolic circumference is an euclidean circumference of center $c_{E}$ and radius $r$ that is parallel to the plane $z = 0$ and $c_{E}$ is in the line generated by $c_{H}$. If $x = (x_{1},x_{2},x_{3})$ is any point in $C$, then $R = d_{\mathbb{H}^{2}}(c_{H},x)$ is given by the formula $cosh(R) = - \langle c_{H},x\rangle = x_{3}$. Since $c_{E}$ is in the line generated by $c_{H}$, then $c_{E} = (0,0,x_{3})$. Since $x$ is in the hyperboloid, we have that $x_{1}^{2}+x_{2}^{2}-x_{3}^{2} = -1$ and since it is in $C$, we have that $x_{1}^{2}+x_{2}^{2} = r^{2}$. So $x_{3}^{2} = 1 + r^{2}$. Using the relation $cosh(R) = x_{3}$, we get that $\cosh^{2}(R) = 1 + r^{2}$, which implies that $\sinh^{2}(R) = r^{2}$ and then $\sinh(R) = r$.

\subsection{Pulleys at the hyperboloid model}

Let $\alpha: [0,s] \rightarrow \mathbb{R}$ be a differentiable map such that $\alpha(0) = 0$ and let $\gamma_{r,\alpha}: [0,s] \rightarrow \mathbb{H}^{2}$ be the differentiable curve defined by $\gamma_{r,\alpha}(t) = r(cos(\alpha(t)),sin(\alpha((t))),\sqrt{\frac{1}{r^{2}}+1})$. It describes the movement over time of the point $p = (r,0,\sqrt{1+r^{2}}))$ inside the circle that is parallel to the plane $z = 0$, with euclidean center $(0,0,\sqrt{1+r^{2}}))$, hyperbolic center $(0,0,1)$, euclidean radius $r$ and hyperbolic radius $R$ such that $\sinh(R) = r$. The linear velocity of $p$ is given by $\gamma_{r,\alpha}'(t) = r\alpha'(t)(-sin(\alpha(t)),cos(\alpha(t)),0)$.

If we take $\lambda$ as the geodesic segment $\overline{(0,0,1),(r,0,\sqrt{1+r^{2}}))}$, then $\alpha_{(0,0,1),\lambda}(\gamma_{r,\alpha}(t)) = \alpha$. This implies that the angular velocity of $\gamma_{r,\alpha}$ is given by $\omega_{(0,0,1)}(\gamma_{r,\alpha}) = |\alpha'|$.

Let $C_{1}$ and $C_{2}$ be two pulleys in $\mathbb{H}^{2}$ and with hyperbolic radii $R_{1}$ and $R_{2}$, respectively. Suppose that both are rotating with the same norm for their linear velocities. Since the linear velocity and the angular velocity are invariant by isometries, we can suppose that $C_{1}$ and $C_{2}$ have hyperbolic center in $(0,0,1)$. The euclidean radii of $C_{1}$ and $C_{2}$ are, respectively, given by the relation $\sinh(R_{i}) = r_{i}$. Then the point $(r_{i},0, \sqrt{1+r_{i}^{2}})$ is in $C_{i}$ and its movement is described by the curve $\gamma_{i} = \gamma_{r_{i},\alpha_{i}}$, for some map $\alpha_{i}$.

The norm of the linear velocity of $C_{1}$ and $C_{2}$ is the same, i.e. for every $t \in [0,s]$:
$$||\gamma_{r_{1},\alpha_{1}}'(t)||_{\gamma_{r_{1},\alpha_{1}}(t)} = ||\gamma_{r_{2},\alpha_{2}}'(t)||_{\gamma_{r_{2},\alpha_{2}}(t)}$$

Then we have:
$$\langle \gamma_{r_{1},\alpha_{1}}'(t), \gamma_{r_{1},\alpha_{1}}'(t) \rangle=  \langle \gamma_{r_{2},\alpha_{2}}'(t), \gamma_{r_{2},\alpha_{2}}'(t) \rangle$$
$$r^{2}_{1}\alpha_{1}'(t)^{2} =  r^{2}_{2}\alpha_{2}'(t)^{2} $$

Since $\sinh(R_{i}) = r_{i}$, it is a positive number and $|\alpha_{i}'(t)| = \omega_{(0,0,1)}(\gamma_{i})(t)$, then:
$$\sinh(R_{1}) \omega_{(0,0,1)}(\gamma_{1})(t) =  \sinh(R_{2})\omega_{(0,0,1)}(\gamma_{2})(t)$$

Since the angular velocity does not depend on the choice of the initial point (\textbf{Proposition \ref{angularvelocitywelldefined}}), then we have another proof of \textbf{Theorem \ref{hyperbolic}}.

\section{The sphere}

Let's consider $\mathbb{S}^{2}$ as the unitary sphere in $\mathbb{R}^{3}$ with its induced Riemannian metric.  Then, if $p \in \mathbb{S}^{2}$ and $u,v \in T_{p}\mathbb{S}^{2}$ (we identify it with its embedding on $\mathbb{R}^{3}$), the inner product is given by $\langle u,v\rangle_{p} =\langle u,v\rangle$, where $\langle,\rangle$ is the regular inner product of $\mathbb{R}^{3}$ and $||\cdot||$ its norm. We denote the norm of $T_{p}\mathbb{S}^{2}$ by $||\cdot||_{p}$.

Let $C$ be a spherical circumference in $\mathbb{S}^{2}$ with spherical center $c_{S}$ and spherical radius $R$. This is an euclidean circumference of center $c_{E}$ and radius $r$. If $x$ is any point in $C$, then $R = d_{\mathbb{S}^{2}}(c_{S},x)$ is the length of the arc between $c_{S}$ and $x$. So $R$ is the angle between $\overline{c_{S},O}$ and $\overline{x,O}$, where $O$ is the center of the sphere (it means that it is equal to the angle $\alpha$ in the figure below). We also have that $\sin{(R)} = \frac{r}{d_{\mathbb{E}^{2}}(x,O)} = r$.

\begin{figure}[h]
\centering
\includegraphics[scale=0.5]{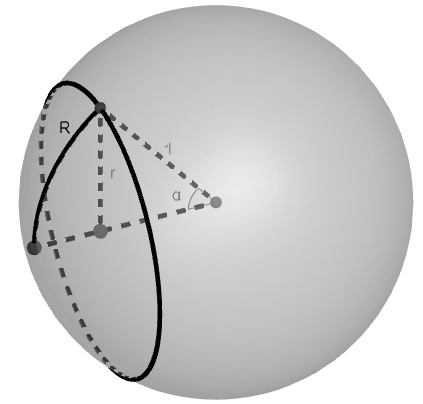}
\includegraphics[scale=0.5]{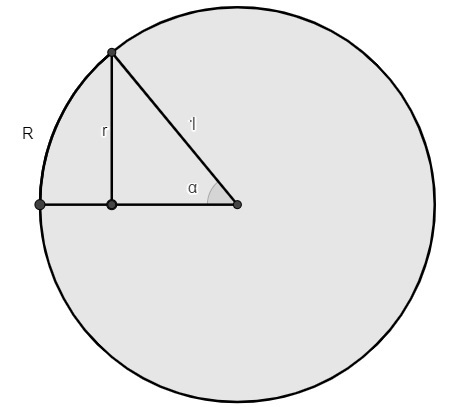}
\caption{\small The second figure represents the plane that contains all triangles of the first one}
\end{figure}

\subsection{Gears at the sphere}

Let $C_{1}$ and $C_{2}$ be two gears attached to each other in the sphere of radius $1$. Suppose that the two gears have, respectively,  $n_{1}$ and $n_{2}$ teeth and spherical radii $R_{1}$ and $R_{2}$, respectively. The movement of the gears' system depends on the width of the teeth, so we assume that $a_{1}(T) = a_{2}(T)$. Notice that a circle $C_{i}$ at the sphere of radius $1$ with spherical radius $R_{i}$ has length $2\pi\sin(R_{i})$.

    Since $n_{i}a_{i}(T)= 2\pi\sin(R_{i})$, then

    \begin{align*}
     \frac{2\pi\sin(R_{1})}{n_{1}} = a_{1}(T) = a_{2}(T) = \frac{2\pi\sin(R_{2})}{n_{2}}
    \end{align*}
    so $$\sin(R_{1})n_{2}=\sin(R_{2})n_{1}.$$

\begin{theorem}\label{sphericalgear} Let $C_{1}$ and $C_{2}$ be two gears attached to each other in the sphere. If the two gears have radii $R_{1}$ and $R_{2}$ and angular velocities are $\omega_{1}$ and $\omega_{2}$, respectively, then we have the following relation:

$$\sin(R_{1}) \omega_{1}(t) =  \sin(R_{2})\omega_{2}(t)$$
\end{theorem}

\begin{proof}We have that $\frac{\sin{R_{1}}}{\sin{R_{2}}} = \frac{n_{1}}{n_{2}}$ and $\frac{n_{1}}{n_{2}} = \frac{\omega_{2}(t)}{\omega_{1}(t)}$ (by \textbf{Theorem \ref{teeth}}), which implies that $\frac{\sin{R_{1}}}{\sin{R_{2}}} = \frac{\omega_{2}(t)}{\omega_{1}(t)}$ and then $\sin(R_{1}) \omega_{1}(t) =  \sin(R_{2})\omega_{2}(t)$.
\end{proof}

\subsection{Pulleys at the sphere}

Let $\alpha: [0,s] \rightarrow \mathbb{R}$ be a differentiable map such that $\alpha(0) = 0$ and $\gamma_{r,\alpha}: [0,s] \rightarrow \mathbb{S}^{2}$ the differentiable curve defined by $\gamma_{r,\alpha}(t) = r(cos(\alpha(t)),sin(\alpha((t))),\sqrt{\frac{1}{r^{2}}-1})$. It describes the movement over time of the point $p = (r,0,\sqrt{1-r^{2}}))$ inside the circle that is parallel to the plane $z = 0$, with euclidean center $(0,0,\sqrt{1-r^{2}}))$, spherical center $(0,0,1)$, euclidean radius $r$ and spherical radius $R \in [0,2\pi]$ such that $\sin(R) = r$. The linear velocity of $p$ is given by $\gamma_{r,\alpha}'(t) = r\alpha'(t)(-sin(\alpha(t)),cos(\alpha(t)),0)$.

If we take $\lambda$ as the geodesic segment $\overline{(0,0,1),(r,0,\sqrt{1-r^{2}}))}$, then $\alpha_{(0,0,1),\lambda}(\gamma_{r,\alpha}(t)) = \alpha$. This implies that the angular velocity of $\gamma_{r,\alpha}$ is given by $\omega_{(0,0,1)}(\gamma_{r,\alpha}) = |\alpha'|$.

Let $C_{1}$ and $C_{2}$ be two pulleys in $\mathbb{S}^{2}$ with spherical radii $R_{1}$ and $R_{2}$, respectively. Suppose that both are rotating with the same norm for their linear velocities. Since the linear velocity and the angular velocity are invariant by isometries, we can suppose that $C_{1}$ and $C_{2}$ have spherical center in $(0,0,1)$. The euclidean radii of $C_{1}$ and $C_{2}$ are, respectively, given by the relation $\sin(R_{i}) = r_{i}$. Then the point $(r_{i},0, \sqrt{1-r_{i}^{2}})$ is in $C_{i}$ and its movement is described by the curve $\gamma_{i} = \gamma_{r_{i},\alpha_{i}}$, for some map $\alpha_{i}$.

The norm of the linear velocity of $C_{1}$ and $C_{2}$ is the same, i.e. for every $t \in [0,s]$:
$$||\gamma_{r_{1},\alpha_{1}}'(t)||_{\gamma_{r_{1},\alpha_{1}}(t)} = ||\gamma_{r_{2},\alpha_{2}}'(t)||_{\gamma_{r_{2},\alpha_{2}}(t)}$$

Then we have:
$$\langle \gamma_{r_{1},\alpha_{1}}'(t), \gamma_{r_{1},\alpha_{1}}'(t) \rangle=  \langle \gamma_{r_{2},\alpha_{2}}'(t), \gamma_{r_{2},\alpha_{2}}'(t) \rangle$$
$$r^{2}_{1}\alpha_{1}'(t)^{2} =  r^{2}_{2}\alpha_{2}'(t)^{2} $$

Since $\sin(R_{i}) = r_{i}$, it is a positive number and $|\alpha_{i}'(t)| = \omega_{(0,0,1)}(\gamma_{i})(t)$, then:
$$\sin(R_{1}) \omega_{(0,0,1)}(\gamma_{1})(t) =  \sin(R_{2})\omega_{(0,0,1)}(\gamma_{2})(t)$$

Since the angular velocity does not depend on the choice of the initial point (\textbf{Proposition \ref{angularvelocitywelldefined}}), then we have the theorem:

\begin{theorem}\label{spherical} Let $C_{1}$ and $C_{2}$ be two pulleys attached by a tensioned belt in $\mathbb{S}^{2}$. If $C_{i}$ has spherical radius $R_{i}$ and angular velocity $\omega_{i}$, then we have the following relation:
$$\sin(R_{1}) \omega_{1}(t) =  \sin(R_{2})\omega_{2}(t)$$ \eod
\end{theorem}

\end{document}